\pdfoutput=1
\ifx\AddToHook\undefined\relax%
\else%
\AddToHook{class/before/asl}{%

}
\fi
\documentclass[jsl]{asl}
\usepackage{amssymb}
\usepackage[T1]{fontenc}
\usepackage{lmodern}
\usepackage{mathrsfs}
\usepackage{microtype}
\usepackage{xcolor}
\usepackage{enumerate}
\usepackage{graphicx}
\usepackage{tikz-cd}
\usepackage[colorlinks,
linkcolor={red!50!black}, citecolor={green!40!black}, urlcolor={blue!50!black}
]{hyperref}

\newtheorem{thm}{Theorem}[section]
\newtheorem{cor}[thm]{Corollary}
\newtheorem{lem}[thm]{Lemma}
\newtheorem{prop}[thm]{Proposition}
\newtheorem*{prop*}{Proposition}
\newtheorem*{fact*}{Fact}
\newtheorem*{claim*}{Claim}

\theoremstyle{definition}
\newtheorem{defi}[thm]{Definition}
\newtheorem*{rem*}{Remark}

\title[Fra\"iss\'e
  limit of finite Heyting algebras]{The automorphism group of the Fra\"iss\'e
  limit of finite Heyting algebras}
\author{Kentar\^o Yamamoto}
\address{The Czech Academy of Sciences\\
Pod Vod\'arenskou v\v{e}\v{z}\'i  271/2\\
Libe\v{n}\\
182 00 Praha\\ The Czech Republic}
\date{}
\email{yamamoto@cs.cas.cz}

\newcommand{\FR}{Fra\"iss\'e}

\newcommand{\aru}[1]{\mathop{\exists #1}}
\newcommand{\zenbu}[1]{\mathop{\forall #1}}

\newcommand{\0}{\mathclap{\phantom{\emptyset}}0}
\IfFileExists{emptysetiszero}{}{\renewcommand{\0}{\emptyset}}
\newcommand{\power}{\mathscr{P}}
\newcommand{\io}[1][]{{}^{\circ_{#1}}}
\newcommand{\Aut}{\mathrm{Aut}}
\newcommand{\il}{\iota_{\hookleftarrow}}
\newcommand{\ir}{\iota_{\hookrightarrow}}
\newcommand{\comext}{\bigsqcup}
\newcommand{\Th}{\mathrm{Th}}
\newcommand{\inv}{{-1}}
\newcommand{\indep}[1][]{%
  \mathrel{
    \mathop{
      \vcenter{
        \hbox{\oalign{\noalign{\kern-.3ex}\hfil$\vert$\hfil\cr
            \noalign{\kern-.7ex}
            $\smile$\cr\noalign{\kern-.3ex}}}
      }
    }\displaylimits_{#1}
  }
}

\DeclareMathOperator{\Diag}{Diag}

\DeclareMathOperator{\ran}{ran}
\DeclareMathOperator{\supp}{supp}

\begin{document}
\begin{abstract}
  Roelcke non-precompactness,
  simplicity, and
  non-amenability of the automorphism group of the
  \FR\ limit of finite Heyting algebras are proved among others.
\end{abstract}
\maketitle

In the present article, we examine the \FR{} limit $L$
of (nontrivial) finite Heyting algebras.
The existence of the model-completion $T^*$ of the theory $T$ of
Heyting algebra
stems from the uniform interpolation theorem for propositional intuitionistic logic,
in which the interpolant of two sentences depends on only one of the two sentences \cite{GHILARDI199727,PITTS_1992}.
The \FR{} limit $L$ is the prime model of $T^*$
and was used to derive an axiomatization of $T^*$ by Darni\`ere~\cite{2018arXiv181001704D}.
The results in the present article complement existing literature
on the automorphism groups of ultrahomogeneous
lattices, e.g., the countable atomless Boolean algebra~\cite{anderson58:_algeb_simpl_certain_group_homeom,TRUSS1989494,Kechris2012}
and the universal distributive lattice~\cite{droste00:_autom_group_univer_distr_lattic},
as our ultrahomogeneous structure is not $\omega$-categorical.

The article is organized as follows:
In the first section,
we recall relevant definitions and fix notation.
In the second section,
we compare the automorphism of $L$ with those of better-known ultrahomogeneous
structures,
especially that of the countable atomless Boolean algebra $B$.
It will be proved that $\Aut(L)$ is not Roelcke precompact
and thus is not realized as the automorphism group of any $\omega$-categorical
structure.
Having established that, we will construct continuous embeddings
of $\Aut(L)$ into $\Aut(B)$.
In the last section,
we will see that $\Aut(L)$ is not amenable
and that $\Aut(L)$ is simple.
The argument used to prove the last claim is
applicable to other \FR{} classes of lattices
with the superamalgamation property,
which is of an independent interest
as it characterizes the vailidity of the Craig interpolation theorem
for a nonclassical logic \cite{10.2307/20016013,Maksimova1977}.

It is an important future task to investigate the combinatorics of
the age $\mathrm{Age}(L)$ of $L$,
in particular about the existence of order expansion of $\mathrm{Age}(L)$
with the Ramsey property and the ordering property,
and the metrizability of $\Aut(L)$.

\section{Preliminaries}
We review an important construction of Heyting algebras
(this material appears in, e.g., Chagrov and Zakharyaschev~\cite{chagrov97:_modal_logic}).
For an arbitrary poset $\mathbb P$,
the poset of upward closed sets, or \emph{up-sets}, of $\mathbb P$
ordered by inclusion has a Heyting algebra structure.
We call this Heyting algebra is the \emph{dual} of $\mathbb P$.
Conversely, if $H$ is a \emph{finite} Heyting algebra,
then one can associate with $H$ the poset $\mathbb P$ of join-prime elements of $H$
with the reversed order.
One can show that the dual of $\mathbb P$ is isomorphic to $H$.

Suppose that $H$ and $H'$ are the duals of $\mathbb P$ and $\mathbb P'$, respectively,
and that $f : \mathbb P \to \mathbb P'$ is \emph{p-morphic}, i.e.,
$f$ is monotonic with
\[
  \zenbu{u \in \mathbb P} \zenbu{v \ge f(u)} \aru{w \ge u} f (w) = u,
\]
then the function $f^*$ defined on $H'$ that maps each up-set with
its inverse image under $f$ is a Heyting algebra homomorphism $H' \to H$.
We call $f^*$ the \emph{dual} of $f$ as well.
If $f$ is injective, then $f^*$ is surjective;
if $f$ is surjective, then $f^*$ is a Heyting algebra embedding.

\hyphenation{super-amalgamation}
Henceforth, $L$ is the \FR\ limit of all finite nontrivial Heyting algebras,
which exists \cite{ghilardi02:_sheav_games_model_compl}.
This structure is \emph{ultrahomogeneous}
in the sense that every isomorphism between
finitely generated substructures,
or members of the \emph{age} $\mathrm{Age}(L)$ of $L$,
extends to an automorphism on $L$.
(Throughout the paper, Heyting algebras are structures in the language
$\{0, 1, \wedge, \vee, \to\}$ unless otherwise stated.)
The strong amalgamation property of the theory $T$ of Heyting algebras was proved by Maksimova~\cite{Maksimova1977};
in fact, her construction establishes the \emph{superamalgamation property}
for the class of finite Heyting algebras.
Recall that a \FR{} class $\mathcal K$
of poset expansions has the superamalgamation property
if for every diagram $A_1 \hookleftarrow A_0 \hookrightarrow A_2$
of inclusion maps in $\mathcal K$,
the amalgamation property of $\mathcal K$ is witnessed by
a diagram $A_1 \hookrightarrow A \hookleftarrow A_2$ of inclusion maps
in such a way that $A_1 \indep_{A_0} A_2$,
where $\indep$ is the ternary relations for subsets of $A$ defined as:
\begin{equation}\label{eq:anchor}
  S \indep[U]T
  \iff \zenbu{a \in S}\zenbu{b \in T}
  \begin{Bmatrix}%
    a \le b &\implies& \aru{c \in U} a \le c \le b\\
    b \le a &\implies& \aru{c \in U} b \le c \le a
  \end{Bmatrix}.
\end{equation}
The superamalgamation property
for the class $\mathcal K$ of finite Heyting algebras
follows from the superamalagamation property for $T$
\cite{Maksimova1977}.
Indeed, let $A_0, A_1, A_2 \in \mathcal{K}$ with $A_0 \subseteq A_1$
($i = 1, 2$).
Consider the quantifier-free sentence $\phi$ with parameters
from $A_1 \cup A_2$ that is the conjunctions of $\bigwedge\mathrm{Diag}(A_1)$,
$\bigwedge\mathrm{Diag}(A_2)$,
and the quantifier-free sentence expressing $A_1 \indep[A_0] A_2$,
where $\Diag(\cdot)$ denotes the diagrams of structures.
By the superamalgamation property for $T$,
we have a model of $T \cup \{\phi\}$.
By the stronger form of the finite model property for Heyting algebras
that is applicable to all quantifier-free formulas~%
\cite{darniere10:_codim_pseud_co_heytin_algeb},
$\phi$ has a model in $\mathcal K$.

We introduce notation naming structures obtained
by the superamalgamation property:
Let $D$ be the diagram \mbox{$B \hookleftarrow A \hookrightarrow C$}
in $\mathrm{Age}(L)$,
where $\mathrm{Age}(L)$ the age of $L$ is
regarded as a category whose morphisms are the embeddings.
The superamalgamation property for $\mathrm{Age}(L)$
gives rise to a subalgebra $\comext D$ of $L$
such that there are embeddings $\il^D : B \hookrightarrow \comext D$
and $\ir^D : C \hookrightarrow \comext D$
with $\il^D(B) \indep_{\ir^D(A)}\ir^D(C)$.
One can show
that $\il^D(B) \setminus \il^D(A)$ and $\ir^D(C) \setminus \ir^D(A)$ are disjoint.

\section{Comparison with known automorphism groups}
In this section, we study the automorphism group of $L$
in relation to those of better-known ultrahomogeneous structures.
First of all,
we find it interesting to see that $\Aut(L)$ is distinct from
the automorphism groups of better-known ultrahomogeneous structures.
In particular,
we will later construct  embeddings of $\Aut(L)$ into the automorphism group
of the countable atomless Boolean algebra,
but the first result of this section implies that they cannot be
topological group isomorphisms.
\newcommand{\qftp}{\mathrm{qftp}}
\newcommand{\tp}{\mathrm{tp}}
Recall that a non-archimedian topological group $G$, such as $\Aut(L)$,
is \emph{Roelcke precompact}
if the set of double cosets $\{VxV \mid x \in G\}$
is finite for every open subgroup $V \le G$
(see, e.g., Tsankov~\cite[p.~534]{tsankov12:_unitar_repres_oligom_group}).

\begin{thm}
  $\Aut(L)$ is not Roelcke precompact.
  \emph{A fortiori},
  $\Aut(L)$ cannot be realized as the automorphism group
  of any countable $\omega$-categorical structure.
\end{thm}
\begin{proof}
  Tsankov~\cite{tsankov12:_unitar_repres_oligom_group}
  showed that a topological group is Roelcke precompact
  if and only if it is the inverse limit of some inverse system
  of oligomorphic permutation groups.
  The second part of the claim follows from the first part and this result.

  Since $L$ is ultrahomogeneous,
  $\Aut(L)$ is not Roelcke precompact
  if and only if there are sequences $(a_i)_{i<\omega}, (b_i)_{i<\omega}$
  of elements of $L$
  such that
  $\tp^L(a_i / \0) = \tp^L(b_i / \0) = \tp^L(a_j / \0) = \tp^L(b_j / \0)$
  for $i, j < \omega$,
  and that $\{\tp^L(a_ib_i / \0)\}$ is infinite.
  Furthermore, since $\Th(L)$ eliminates quantifiers,
  types realized in $L$  are in one-to-one correspondence
  with quantifier-free types realized in $L$.
  Finally, as $L$ is locally finite,
  the latter
  are essentially isomorphism types of subalgebras of $L$
  with distinguished generators.

  With that in mind,
  let $F_1$ be the free Heyting algebra whose generator is $x$.
  For each term $t(x) \in F_1$,
  write $L_t$ for the quotient of $F_1$ by the principal filter $\theta_t$
  generated by $t$.
  Furthermore, let $L_t^*$ be the Heyting algebra obtained
  by adding a new minimum element $0^{L_t^*}$ below $0^{L_t}$.
  For a term $t(x)$,
  we define $t^*(x, y)$ to be the term obtained
  by replacing every occurrence of $0$ with $y$.
  One can check that $(t^*)^{L_t^*}([x]_t, 0^{L_t}) = [t(x)]_t \in L_t^*$,
  where $[\cdot]_t$ denotes the congruence class with respect to $\theta_t$.
  Therefore, $L_t^*$ is generated by $0^{L_t}$ and $[x]_t$.
  We have obtained 2-generated subalgebras of $L$ of infinitely many
  isomorphism types.
  On the other hand, we have
  $\langle [x]_t \rangle^{L_t^*} = \langle 0^{L_t} \rangle^{L_t^*}$
  is a 3-chain for all $t$.
\end{proof}

It is well known that $\Aut(M)$ for a countable $\omega$-categorical $M$
is not locally compact \cite{macpherson11:_survey_homog_struc}.
\begin{prop}
  The topological group $\Aut(L)$ is not locally compact.
\end{prop}
\begin{proof}
  It suffices to show that for every finite subset $S \subseteq L$
  there is an infinite orbit in the action of $\Aut(L)_{(S)}$ on $ L$.
  Note that for every finite subalgebra $A \subseteq L$,
  there exists $a \in L \setminus A$
  such that $a$ is join-prime in $\langle Aa \rangle^L$.
  Indeed, consider the dual $\Bbb P$ of $A$ and
  the disjoint union $\Bbb P' := \Bbb P \sqcup \{w\}$,
  where $w$ is a fresh element,
  and let $a$ be the image of $\{w\}$ under the embedding of
  the dual of $\Bbb P'$ into $L$ that fixes $A$ pointwise.
  By repeatedly using this,
  take an $\omega$-sequence $(a_i)_{i<\omega}$ of elements of $L$
  such that $a_i \in L \setminus \langle Sa_0a_1\dots a_{i-1}\rangle^L$
  is join-prime in $\langle Sa_0a_1\dots a_{i}\rangle^L$
  for $i < \omega$.
  By construction, there exists an automorphism $\phi_i : L \to L$
  fixing $S$ pointwise
  such that $\phi_i(a_i) = a_{i+1}$ for $i<\omega$.
  Hence, the orbit of $a_0$ under $\Aut(L)_{(S)}$ is infinite.
\end{proof}

An obvious strategy to study $\Aut(L)$ is to relate it to
$\Aut(B)$, where $B$ is the countable atomless Boolean algebra.
The following lemma gives rise to a topological embedding of the former
into the latter.
Recall that an \emph{interior operator} on a Boolean algebra $B$
is a function from $B$ to $B$
that is decreasing, monotonic, idempotent, and commuting over meets.
We write interior operators in superscripts so that $B^\circ$
is the image of an interior operator ${}^\circ : B \to B$.
For every interior operator ${}^\circ : B \to B$,
the image $B^\circ$ with the induced order
is isomorphic to some Heyting algebra
(see, e.g., \cite{blok76:_variet_inter_algeb}).
\begin{lem}\mbox{}\label{lem:automorphism-group}
  \begin{enumerate}
  \item
    Let $f:H \to H_1$ be a Heyting algebra homomorphism between finite algebras.
    There are finite Boolean algebras $B(H)$ and $B(H_1)$,
    interior operators $\io, \io[1]$ on $B(H), B(H_1)$, respectively,
    and a unique Boolean algebra homomorphism $B(f) : B(H) \to B(H_1)$
    such that $B(H)\io{} \cong H$, $B(H_1)\io[1] \cong H_1$
    and that $B(f)$ extends $f$.
    If $f$ is injective, so is $B(f)$;
    if $f$ is surjective, so is $B(f)$.
  \item
    There is an interior operator $\io$
    on the countable atomless Boolean algebra $B$ such that
    $B\io$ is isomorphic to
    the universal ultrahomogeneous countable Heyting algebra $L$.
  \end{enumerate}
\end{lem}
\begin{proof}\nopagebreak\mbox{}\nopagebreak
  \begin{enumerate}\nopagebreak
  \item\nopagebreak
    Let $P$ and $P_1$ be the dual posets of $H$ and $H_1$, respectively.
    There is a p-morphism $D(f) : P_1 \to P$ that is  the dual of $f$.
    $D(f)$ is surjective if $f$ is injective.
    Let $B(H) = \power (P)$ and $B(H_1) = \power(P_1)$.
    $D(f)$ induces a Boolean algebra homomorphism $B(f) : B(H) \to B(H_1)$.
    $B(f)$ is injective if $D(f)$ is surjective.
    Likewise, $B(f)$ is surjective if $f$ is.
    Let $\io{}, \io[1]$ be the operations that take a subset
    to the maximal up-set
    contained by that set.
  \item
    Let $(L_i)_{i<\omega}$ be a chain of finite Heyting algebras
    used in the construction of $L$; so $\bigcup_i L_i = L$.
    Let $B_i = B(L_i)$ as above and $\io[i]$ be an interior operator
    such that $B_i\io[i] \cong L_i$.
    We may take $B_i \subseteq B_{i+1}$ for $i < \omega$.
    Then $\io[i+1]$ extends $\io[i]$.
    Let $B = \bigcup_i B_i$ and $\io = \bigcup \io[i]$.
    Then $B\io = \left( \bigcup_i B_i \right )\io =
    \bigcup_i B_i\io[i] = \bigcup_i L_i = L$.
    It remains to show that $B$ is atomless.
    Take an arbitrary $a \in B$ that is nonzero.
    Take $i < \omega$ such that $a \in B_i$.
    Let $P_i$ be the poset dual to $L_i$;
    then $a$ is a nonempty subset of $P_i$.
    Take some $w \in a$.
    Let $P'$ be the poset obtained from $P_i$
    by replacing $w$ with the 2-chain $\{w_1 < w_2\}$.
    Let $\pi: P' \twoheadrightarrow P_i$ be the surjection
    that maps the chain to $\{w\}$ and is the identity elsewhere.
    This is a p-morphism,
    and it induces $\iota: L_i \hookrightarrow L'$,
    where $L'$ is the dual of $P'$.
    Take $k < \omega$ such that
    there is an embedding $\iota': L' \hookrightarrow L_k$
    such that $\iota' \circ \iota$ is the identity on $L_i$.
    Let $b = (a \setminus \{w \}) \cup \{w_1\}$.
    Then $b \in B_k = B(L_k) \subseteq B$ and $0 < b < a$.
    \hfil\qed
  \end{enumerate}
  \noqed
\end{proof} 

\begin{thm}
    An automorphism $L \to L$ can be extended (as a function between pure sets)
    to an automorphism $B \to B$.
    This extention is unique.
    Moreover, this defines an injective group homomorphism
    $\mathrm{Aut}(L) \hookrightarrow \mathrm{Aut}(B)$
    that is a homeomorphism onto its image.
\end{thm}
\begin{proof}
    Let $f : L \to L$ be an automorphism.
    Let $f_k : L_k \to L_k'$ be the restriction of $f$ to $L_k$
    where $L_k' = f(L_k)$.
    Each $f_k$ is an isomorphism.
    By the fact above,
    $f_k$ induces a Boolean algebra isomorphism
    $B(f_k) : B(L_k) \to B(L_k')$ for each $k < \omega$;
    and by construction $B(f_j)$ extends $B(f_k)$ for each $k < j < \omega$.
    Let $\hat f = \bigcup_k B(f_k)$.
    Then $\hat f$ is an isomorphism $B \to B$.

    Let $g : L \to L$ be another isomorphism.
    We have $\hat f \circ \hat g = (f \circ g){\hat{}}$
    because each side of the equation extends $f \circ g$.
    
    Let $\iota : \mathrm{Aut}(L) \to \mathrm{Aut}(B)$ be the map
    $f \mapsto \hat f$.  The map $\iota$ is a group homomorphism as seen above,
    and it is clearly injective.
    
    Next, we prove that $\iota$ is continuous.
    Let  $\bar b$ be a tuple in $B$.
    It suffices to show that 
    for an automorphism $f : L \to L$  the value of $\hat f(\bar b)$
    is determined by the value of $f(\bar a)$
    for  a tuple $\bar a$ in $L$.
    There exists $k < \omega$ such that $\bar b$ is in $B_k = B(L_k)$.
    Let $f_k : L_k \to L_k'$ be an isomorphism 
    that is a restriction of $f$.
    Then $\hat f(\bar b) = B(f_k)(\bar b)$.
    Let $\bar a$ be an enumeration of the finite algebra $L_k$;
    then $\bar a$ is what we needed.

    Finally, we show that the image $\iota(U)$ is open
    in $\operatorname{ran} \iota \subseteq \mathrm{Aut}(B)$
    for an arbitrary basic open set $U$ of
    $\mathrm{Aut}(L)$.
    Indeed, let $U$ be the set of $f : L \to L$ fixing the values of $f$
    at $\bar a \in L$;
    then $\hat g \in \iota(U)$ in and only if
    $\hat g \upharpoonright B_0 = \hat f \upharpoonright B_0$ for $g : L \to L$,
    where $B_0$ is the Boolean subalgebra of $B$ generated by $\bar a$.  
\end{proof}

Note that 
the structure $L$ is not interpretable in $B$ because
the latter is $\aleph_0$-categorical whereas the former is not.
 
There is another way $\Aut(B)$ and $\Aut(L)$ can be related.
Recall that a \emph{relativized reduct} is a special sort of interpretation
where the domain of the interpreted structure is a $0$-definable subset
of the domain (as opposed to powers thereof) of the interpreting structure.
\begin{lem}\label{lem:joinofregular}
  There is an atomless Boolean algebra which is a relativized reduct
  $B$ of $L$, where
  every element $L$ is a finite join of elements of $B$.
\end{lem}
\begin{proof}
  The set $B$ of fixed points of $1 - (1 - \cdot)$ in $L$ is a Boolean algebra
  by setting $a \wedge^B b = \neg\neg(a \wedge^L b)$
  and the remaining operations of $B$
  the restrictions of the corresponding operations of $L$.
  (Note that $B$ is not a substructure of $L$.)

  Suppose that $a \in B$ is an atom of $B$.
  We show that $a$ is also an atom of $L$.
  To see this, assume the contrary,
  and let $b$ be such that $0 < b < a$, where $b \not \in B$.
  Since $b \not \in B$, we have $1 - (1 - b) \neq b$;
  since $1 - (1 - c) \le c$ for all $c \in B$, we have $1 - (1 - b) < b$.
  Now $1 - (1 - b) \in B$ and $0 < 1 - (1 - b)$ (since $1 - b < 1$),
  so we have $0 < 1 - (1 - b) < a$,
  contradicting the assumption that $a$ is an atom of $B$.

  We have seen that any atom in $B$ is an atom of $L$.
  Since there are no join-irreducible elements (let alone atoms) in $L$
  \cite[Proposition 4.28.(iii)]{ghilardi02:_sheav_games_model_compl},
  $B$ is atomless.

  Let $a \in L$ be arbitrary.
  Take a finite subalgebra $H \subseteq L$ such that $a \in H$,
  and let $\mathbb P$ be the dual poset of $H$
  so we may identify an element of $H$ with an up-set of $\mathbb P$.
  Possibly by replacing $L$ by another finite Heyting algebra into which
  $L$ embeds,
  we may assume that $\mathbb P$ is a forest.
  Furthermore,
  without loss of generality, we may assume that $a$ is principal
  as a subset of $\mathbb P$,
  generated by $x \in \mathbb P$.
  If $x$ is a root, then $a$ itself is regular, so there remains nothing
  to be shown.
  Suppose not,
  and let $x^-$ be the predecessor of $x$.
  Let $\mathbb P_1, \mathbb P_2$ be disjoint posets isomorphic
  to that induced by $a \subseteq \mathbb P$.
  Let $\mathbb P' := (\mathbb P \setminus a)
  \sqcup \mathbb P_1 \sqcup \mathbb P_2$
  whose partial order is the least containing those of the summands 
  and $x^- \le \mathbb P_1$, $x^- \le \mathbb P_2$.
  Consider the surjective p-morphism $\mathbb P' \twoheadrightarrow \mathbb P$
  that collapses $\{\min \mathbb P_1, \min \mathbb P_2\}$ to $x$,
  and let $i :H \hookrightarrow H'$ be the Heyting algebra embedding it induces.
  (See also Figure~\ref{fig:mine}.)
  Note that $\mathbb P_i \in H'$ is regular for $i = 1, 2$
  and that $i(a) = \mathbb P_1 \vee \mathbb P_2$.
  Let $H_r(a)$ be a subalgebra of $L$ such that
  there is an isomorphism $\phi : H' \to H_r(a)$
  that extends the identity map on $H$.
  Let $r_1(a) := \phi(\mathbb P_1)$ and $r_2(a) := \phi(\mathbb P_2)$.
  We have $a = r_1(a) \vee r_2(a)$
  and $r_i(a) \in B$ $(i = 1,2)$ as promised.
  \begin{figure}[htbp]
    \centering{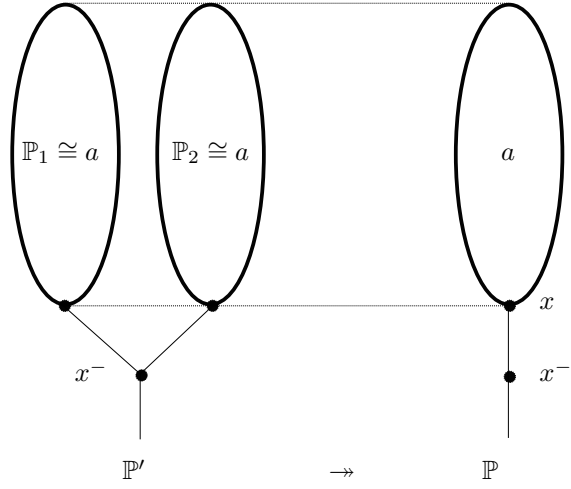}
    \caption{Construction of $H'$.}
    \label{fig:mine}
  \end{figure}
\end{proof}

\hyphenation{homeo-morphism}
\begin{prop}
  Let $h_{\neg\neg} : \Aut(L) \to \Aut(B)$ be the continuous homomorphism
  induced by the interpretation of the Lemma above.
  This is injective
  and is a homeomorphism onto its image.
  However, $h_{\neg\neg}$ is not surjective, and 
  its image is a non-dense non-open subset of $\Aut(B)$.
\end{prop}
\begin{proof}
  The first claim is immediate.
  We show that $h_{\neg\neg}$ is not surjective.

  Consider the 3-element chain $C_3$,
  which can be regarded as a Heyting algebra,
  and let $a \in C_3$ be such that $0 < a < 1$.
  Note that $a$ is irregular and
  a principal up-set in the dual finite poset of $C_3$.
  Let $D$ be the diagram $C_3 \hookleftarrow \mathbf 2 \hookrightarrow C_3$,
  where $\mathbf 2$ is the 2-element Heyting algebra.
  Let $a_0 = \il^D(a)$, $a_{1.5} = \ir^D(a)$, and $H = H_r(a_{1.5})$.
  Next, let $D'$ be the diagram
  $H \hookleftarrow \il^D(C_3) \hookrightarrow H$.
  Let $a_{1i} = \il^{D'}(r_i(a_{1.5}))$, $a_{2i} = \ir^{D'}(r_i(a_{1.5}))$
  and $a_{0i} = r_i(a_0)$ for $i=1,2$.
  Refer to Figure~\ref{fig:caitlin} for this construction.
  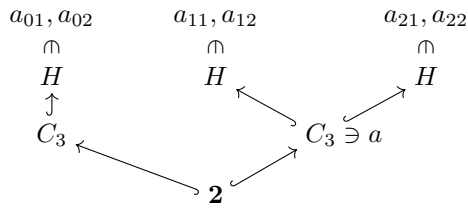
\begin{figure}[htbp]
    \newcommand\lin{\protect\rotatebox{270}{$\in$}}
    \centering
    \begin{tikzcd}[cramped, sep=small]
      a_{01}, a_{02} \ar[d, phantom, "\lin"] & & a_{11}, a_{12} \ar[d, phantom, "\lin"]& & a_{21}, a_{22} \ar[d, phantom, "\lin"]\\
      H & & H & & H \\
      C_3 \ar[u, hook ]& & & C_3 \ar[ul, hook] \ar[ur, hook] \ar[r, phantom, "\ni a",  near start]& {}\\
      & & \mathbf{2} \ar[ull, hook] \ar[ur, hook]
    \end{tikzcd}
    \caption{Construction by amalgamation}
    \label{fig:caitlin}
  \end{figure}
  
  The Boolean subalgebra $B_6$ generated by $a_{ji}$
  ($0 \le j \le 2, 1 \le i \le 2$)
  in $B$
  has six atoms,
  each permutation of which extends to an automorphism of $B$.
  Consider the permutation $a_{ji} \mapsto a_{(j + 1 \bmod 3)i}$,
  which extends to an automorphism of $B_6$,
  which in turn extends to $\phi \in \Aut(B)$ by ultrahomogeneity of $B$.
  By construction,
  \[\bigvee_L \phi(\{a_{11}, a_{12}\}) \neq \bigvee_L \phi(\{a_{21}, a_{22}\})\]
  showing that $\phi$ is not in the image of $h_{\neg\neg}$.

  The last paragraph also shows that the image of $h_{\neg\neg}$ is not dense.
  To see that $\ran h_{\neg\neg}$ is not open,
  let $\overline b$ be an arbitrary tuple in $B$,
  and we prove that $\Aut(B)_{(\overline b)} \setminus \ran h_{\neg\neg} \neq \0$.
  Take  a finite subalgebra $H$ of $L$
  such that $H$ generates $\langle \overline a\rangle^B$ as a Boolean algebra.
  Let $D''$ be the diagram%
  \footnote{To be more precise, one can replace $\bigsqcup D$ by
  an appropriate copy by the weak homogeneity of $L$.}
  $H \hookleftarrow \mathbf 2 \hookrightarrow \bigsqcup D$.
  The image $\ran \ir^{D''}$ generates a copy $B'_6$ of $B_6$.
  Take an automorphism $\psi_0$ on $\bigsqcup D''$
  $\psi_0 \upharpoonright B'_6$ is as constructed in the preceding paragraph and
  that $\psi_0 \upharpoonright \ran \il^{D''}$ is the identity.%
  \footnote{The existence of such an automorphism can be proved in terms of the concrete representation of the $\bigsqcup D''$.}
  The automorphism $\psi_0$ extends to another $\phi \in \Aut(B)$,
  which is in $\Aut(B)_{(\overline b)} \setminus \ran h_{\neg\neg}$.
\end{proof}

\section{Amenability and simplicity}
We now proceed to showing the non-amenability of $\Aut(L)$.
\newcommand{\alex}{\mathrm{alex}}
\begin{defi}
  Let $H$ be a finite nondegenerate Heyting algebra.
  For $b \in H$, we write $I(b)$ for the set of join-prime elements below or equal to $b$.
  Let $\prec$ be an arbitrary linear extension of the partial order on $I(1)$
  induced from $H$.
  We define a total order $\prec^\alex$ on $H$ extending $\prec$
  by the following:
  \[
    a \prec^\alex a' \iff \max_\prec (I(a) \mathbin{\triangle} I(a')) \in I(a').
  \]
  This is clearly a total order, which is
  known as the anti-lexicographic order.
  We call this a \emph{natural ordering} on $H$.
  
  An expansion of a finite nondegenerate Heyting algebra $H$ by a
  natural total order
  is called a \emph{finite Heyting algebra with a natural ordering}.
\end{defi}

It is easy to check that if 
$(H, \prec)$ is a finite Heyting algebra with a natural ordering,
and $H$ happens to be a Boolean algebra,
then $(H, \prec)$ is a finite Boolean algebra with a natural ordering
in the sense of Kechris, Pestov, and Todor\v{c}evi\'c~\cite{kechris05:_fra_ramsey_theor_topol_dynam_autom_group}.
Recall from  the same paper that an order expansion $\mathcal C^*$ of a \FR{}
class $\mathcal C$ is \emph{reasonable}
if
there exists some \emph{admissible} order $\prec_2$ on $M_2$,
i.e., an order such that $(M_2, \prec_2) \in \mathcal{C}^*$,
which extends $\prec_1$
whenever $M_1, M_2 \in\mathcal C$ with $M_1$ a subalgebra of $M_2$,
and $\prec_1$ is admissible on $M_1$.

\begin{prop}
  The class $\mathcal K^*$ of finite Heyting algebras with a natural ordering
  is a reasonable \FR\ expansion of $\mathrm{Age}(L)$.
\end{prop}
\begin{proof}
  We show that $\mathcal K^*$ is reasonable and
  that $\mathcal K^*$ has the amalgamation property.
  (Other claims are clear.)
  In what follows, for a totally ordered set $(X, <)$ and $Y, Z \subseteq X$,
  we write $Y < Z$ to mean that $y < z$ whenever $y \in Y$ and $z \in Z$.

  Let $H_1 \subseteq H_2$ be finite Heyting algebra, and
  let $\prec_1^\alex$ be an arbitrary admissible total order on $H_1$.
  We show that there exists an admissible order on $H_2$ extending $\prec_1^\alex$.
  Let $\pi: \mathbb P_2 \twoheadrightarrow \mathbb P_1$ be the surjective p-morphism dual to the inclusion map $H_1 \hookrightarrow H_2$.
  Note that with  $I(1_{H_i})$ and $\mathbb P_i$ identified as pure sets,
  an admissible total order of $H_i$ extends the dual of the order of $\mathbb P_i$ for $i = 1,2$.
  
  Suppose that for $p,q \in \mathbb P_1$ we have $p \prec_1 q$.
  Since $\prec_1^\alex$ is admissible, $p \not \le q$.
  Take arbitrary $p', q' \in \mathbb P_2$ such that $\pi(p') = p$
  and that $\pi(q') = q$.
  Since $\pi$ is order-preserving \emph{a fortiori},
  we have $p' \not \le q'$.

  Let $R = (\mathord{\le} \setminus \Delta)
  \cup \{(p',q') \mid \pi(p') \prec_2 \pi(q')\}$ be a binary relation on
  $\mathbb P_2 = I(1_{H_2})$,
  where $\Delta$ is the diagonal relation.
  It can be shown by induction from the fact in the preceding paragraph
  that $R$ contains no cycle.
  Therefore, $R$ can be extended to a total order $\prec_2$.
  Furthermore, for $p, q \in \mathbb P_1$, 
  we have $\pi^{-1}(p) \prec_2 \pi^{-1}(q)$;
  \emph{a fortiori}, $\pi^{-1}(p) \prec_2^\alex \pi^{-1}(q)$.
  This shows that $\prec_2^\alex$ extends $\prec_1^\alex$.

  Next, we prove the amalgamation property for $\mathcal K^*$.
  Let $D$ be the diagram
  $H_1 \hookleftarrow H_0 \hookrightarrow H_2$
  in $\mathrm{Age}(L)$
  and let $\prec_i^\alex$ be an arbitrary admissible ordering on $H_i$ for $i = 1,2$.
  Recall the dual poset $\mathbb P$ of $\comext D$
  is a sub-poset of the product order $\mathbb P_1 \times \mathbb P_2$,
  where $\mathbb P_i$ is the dual of $H_i$ ($i = 1, 2$) \cite{Maksimova1977}.
  Define a total order $\prec$ on $\mathbb P$
  so it extends the product order of $\prec_1$ and $\prec_2$.

  We first show that $\prec$ extends the dual of the order of $\mathbb P$.
  Assume that $(p_1,p_2) \le (q_1, q_2)$ for $(p_i,q_j) \in \mathbb P$
  and $1 \le i, j \le 2$.
  (Recall that $p_i, q_i \in \mathbb P_i$.)
  Since the order of $\mathbb P$ is induced by the product of those of
  $\mathbb P_1$
  and $\mathbb P_2$,
  we have $p_i \le q_i$ for $i = 1,2$.
  Because $\prec_i$ extends the dual of the order of $\mathbb P_i$,
  we have $p_i \succ_i q_i$ ($i = 1, 2$).
  By the construction of $\prec$, we have $(p_1,p_2) \succ (q_1, q_2)$
  as desired.

  We then
  prove that $(\comext D, \prec^\alex)$  witnesses the amalgamation property.
  Because of the strong amalgamation property of $\mathrm{Age}(L)$,
  it suffices to show that $\prec^\alex$ extends $\il^D(\prec_1^\alex)$
  and $\ir^D(\prec_2^\alex)$.
  Take $p, p' \in \mathbb P_1$, and assume that $p \prec p'$
  (the other case can be handled in a similar manner).
  Since $\il^D$ is induced by the projection
  $\pi_1 : \mathbb P \twoheadrightarrow \mathbb P_1$,
  it suffices to show that $\pi^{-1}(p) \prec^\alex \pi^{-1}(p')$.
  Now, it is easy to see that, in fact, $\pi^{-1}(p) \prec \pi^{-1}(p')$
  by the construction of $\prec$.
\end{proof}

\begin{cor}\mbox{}
   $\Aut(L)$ is not amenable.
\end{cor}
\begin{proof}
  We will make use of the following proposition:
  \begin{prop*}[{\cite[Proposition 2.2]{Kechris2012}}]
    Let $\mathcal C$ be a \FR{} class
    and $\mathcal C^*$ a \FR{} order expansion of $\mathcal C$
    that is reasonable and has the ordering property.
    Moreover,
    suppose that there are $A, B \in \mathcal C$
    and an embedding $\iota_< : A \to B$
    for each admissible ordering $<$ on $A$
    with the following properties:
    \begin{enumerate}[(i)]
    \item There is an admissible ordering $<'$ on $B$
      such that for every admissible ordering $<$ on $A$,
      the function $\iota_<$ does not embed $(A, <)$ into $(B, <')$;
    \item
      For any two distinct admissible orderings $<_1$, $<_2$ on $A$,
      there exists an admissible ordering $<'$ on $B$
      such that at least one of $\iota_{<_1}$ and $\iota_{<_2}$
      fails to embed $(A, <_1)$ or $(A, <_2)$, respectively,
      into $(B, <')$.
    \end{enumerate}
    Then, the automorphism of the \FR{} limit of $\mathcal C$
    is not amenable.
  \end{prop*}
  Consider the following construction appearing in 
  \cite[Remark 3.1]{kechris05:_fra_ramsey_theor_topol_dynam_autom_group}.
  Let $A$ be the finite Boolean algebra with the atoms $a$ and $b$
  and $B$ with $x$, $y$, and $z$.
  For the order $<_1$ which extends $a <_1 b$, define:
  \begin{align*}
    \pi_{<_1}(a) := x,&& \pi_{<_2}(b) := y \vee z
  \end{align*}
  Moreover, for the order $<_2$ which extends $b <_2 a$:
  \begin{align*}
    \pi_{<_2}(a) := y, && \pi_{<_2}(b) := x \vee z
  \end{align*}
  Let $<'$ be defined as extending $z <' y <' x$.
  The objects defined above 
  witness the conditions (i) and (ii).
  We conclude that $\Aut(L)$ is not amenable.
\end{proof}

Finally, we study the aspects of the combinatorics of $\mathrm{Age}(L)$
pertaining to the extreme amenability of $\Aut(L)$.
The Kechris-Pestov-Todor\v{c}evi\'c correspondence concerns
order expansions of the ages of ultrahomogeneous structures with the ordering
property \cite{kechris05:_fra_ramsey_theor_topol_dynam_autom_group}.
One can make an empirical observation that
many arguments establishing the ordering property of an
order expansion of a \FR\ class
fall into two categories:
one based on a lower-dimensional Ramsey property
and the other rather trivially using the order-forgetfulness
of the expansion.
The former is applied to many classes of relational structures such as graphs,
whereas the latter is used with the countable atomless Boolean algebras
and the infinite-dimensional vector space over a finite field.
Our structure $L$ is similar to the latter classes of structures.
However, we see the following.
\begin{prop}
  There is no \FR\ order class of isomorphism types
  that expands the class of finite Heyting algebras and is order-forgetful.
\end{prop}
\begin{proof}
  Suppose that such a class $\mathcal K^*$ exists.
  Let $H$ be an arbitrary finite Heyting algebra,
  and consider the action of $\Aut(H)$ on the set of binary relations on $H$.
  Since $\mathcal K^*$ is closed under isomorphism types,
  the set of admissible orderings $A_L$ on $H$ is a union of orbits.
  Since $\mathcal K^*$ is order-forgetful, $A_L$ consists of a single orbit.

  Now consider the poset $\mathbb P'$ that is
  the disjoint union of two 2-chains,
  with its quotient $\mathbb P$ obtained by collapsing one of the 2-chains
  into a point.
  The canonical surjection $\mathbb P' \twoheadrightarrow \mathbb P$
  is p-morphic,
  which induces a Heyting algebra embedding $H \hookrightarrow H'$.
  Let $a, b \in H'$ correspond to the two 2-chains.
  Clearly, $H$ is rigid whereas there is an automorphism $\phi : H' \to H'$
  under which $a$ and $b$ are conjugates.
  Consider an admissible ordering $\prec$ on $H'$;
  without loss of generality, we may assume $a \prec b$.
  Writing the action of $\Aut(H')$ by superscripts, we have $b \prec^\phi a$.
  Since $\mathcal K^*$ is a \FR\ class,
  the restrictions of $\prec$ and $\prec^\phi$ to $H$, respectively,
  are admissible orderings on $H$.
  Now, we have $\mathord{\prec} \cap H^2 \neq \mathord{\prec^\phi} \cap H^2$,
  as witnessed by $(a,b)\in H^2$.
  These cannot belong to the same orbit of $A_H$ as $H$ is rigid.
\end{proof}

From this point on, we study $\Aut(L)$ as an abstract group,
and show that it is simple.
Our argument is based on the technique by Tent and
Ziegler~\cite{tent11:_urysoh}.
Our ternary relations is reminiscient of
the stationary independence relation on the random poset defined in Calderoni, Kwiatkowska, and Tent~\cite{CALDERONI202143},
but note that our setting is different as our language is algebraic.
In what follows, recall that structures
$M_1$ and $M_2$ sharing the same domain $M$ but of possibly different languages
are \emph{definitionally equivalent}
if subsets of $M^n$ are $0$-definable in $M_1$ if and only if it is
$0$-definable in $M_2$ for every $n < \omega$.

\begin{lem}\label{lem:stationary}
  Let $M$ be a countable ultrahomogeneous structure.
  Suppose that $M$ definitionally equivalent expansion of a bounded semilattice.
  If $\mathrm{Age}(M)$ has the superamalgamation property,
  then the ternary relation $\indep'$ among finite sets of $M$
  defined by
  \[
    A\indep[B]'C \iff \langle AB \rangle \indep[\langle B\rangle] \langle BC \rangle,
  \]
  where $\indep$ is as in formula~(\ref{eq:anchor}),
  and $\langle S \rangle$ denotes substructure generated by $S$,
  is a stationary independence relation
  in the sense of Tent and Ziegler~\cite{tent11:_urysoh}.
\end{lem}
\begin{proof}
  By \cite[Examples~2.2.1]{tent11:_urysoh},
  the ternary relation~$\indep'$ satisfies the axioms Existence, Invariance, and Stationarity.
  By the shape of the definition of $\indep'$,
  we have Monotonicity and Symmetry.
  It remains to show the axiom Transitivity.
  Since $M$ is ultrahomogeneous and
  definitionally equivalent to its partial order reduct,
  whenever $A, B, C \subseteq M$ are finite,
  any $S$ witnessing the superamalgamation property for
  the diagram $\langle AB \rangle \hookleftarrow \langle B \rangle \hookrightarrow \langle BC \rangle$
  belong to the same $\Aut(M)_{(B)}$-orbit.
  This orbit only depends on the orbit of $A$ and that of $B$.
  Moreover, $A \indep[B]' C$ if and only if the $\Aut(M)_{(B)}$-orbit $O$ of $AC$
  is such that the orbit of $S$ above is that of $\langle ABC \rangle$,
  which only depends on $O$.
  That $A \indep[BC]' D$ and $A \indep[B]' C$ imply $A \indep[B]' D$ follows easily from this characterization.
\end{proof}

\begin{thm}\label{thm:simple}
  Let $M$ be as in the Lemma.
  If $\mathrm{Age}(M)$ has the superamalgamation property,
  and $\le$ is dense,
  then the abstract group $\Aut(M)$ is simple.
  In fact, for any nontrivial $g \in \Aut(M)$,
  every element of $\Aut(M)$ is the product of at most $16$ conjugates
  of $g$ and $g^\inv$.
\end{thm}
\begin{proof}
  Let $g \in \Aut(M)$ be nontrivial.
  Then $\supp g$ is infinite by density.
  Indeed, take $a \in M$ such that $g(a) \neq a$;
  the interval $(a \wedge g(a), a)$, which is infinite by density,
  is included in $\supp g$.
  One can then see that there is no type over a finite set
  whose set of realizers is infinite and fixed pointwise by $g$.
  This follows from $|\supp g| = \aleph_0$ by arguing in the same manner
  as in \cite[Corollary 2.11]{MACPHERSON201140}.
  Finally, since if $A\indep[X]'B$, $X' \subseteq X$, and $(A\cup B)\cap X \subseteq X'$, then $A\indep[X']'B$,
  by \cite[Lemma 5.1]{tent11:_urysoh}, the claim follows.
\end{proof}

\begin{cor}
  $\Aut(L)$ is simple.
  In fact, for any nontrivial $g \in \Aut(L)$,
  every element of $\Aut(L)$ is the product of at most $16$ conjugates
  of $g$ and $g^\inv$.
\end{cor}

By the result by Maksimova,
our argument shows the simplicity of the automorphism group
of the \FR{} limit of all finite members of each of the 7 nontrivial subvarieties of
Heyting algebras with the (super-)amalgamation property.

\subsection*{Acknowledgment}
The author benefited from discussions with (in no particular order)
Dana Barto\v{s}ov\'a, Petr Cintula, Reid Dale,
Alexander Kruckman, Igor Sedl\'ar, Pierre Simon,
Johann Wannenburg.
The comments by the anonymous reviewer were immensely helpful.
Last, but not least, the author express his utmost gratitute to
Dugald Macpherson, who pointed out the error in the original proof of Theorem~\ref{thm:simple}.
The author was financially supported by Takenaka Scholarship Foundation,
University of California,
and the Czech Academy of Sciences during the preparation of the present article.

\bibliographystyle{asl}
\bibliography{thebib}

\end{document}